\documentclass[a4paper,12pt]{article}

\usepackage[english]{babel}
\usepackage{amsmath,amsfonts,amssymb,amsthm,bbm}

\usepackage[top=2cm, bottom=2cm, left=2.5cm, right=2.5cm]{geometry}

\usepackage{indentfirst}
\usepackage[T1]{fontenc} 

\usepackage{graphicx}
\usepackage{float}

\newtheorem{theorem}{Theorem}[section] 
\newtheorem{lemma}[theorem]{Lemma}     
\newtheorem{definition}[theorem]{Definition}
\newtheorem{corollary}[theorem]{Corollary}

\numberwithin{equation}{section}

\renewcommand{\H}{\mathbb H}
\newcommand{\N}{\mathbb N}

\newcommand{\R}{\mathbb R}

\newcommand{\Fi}{\mathbf \Phi}
\newcommand{\Ffi}{\mathbf \phi}
\newcommand{\x}{\mathbbm x}
\newcommand{\y}{\mathbbm y}
\newcommand{\z}{\mathbbm z}
\newcommand{\e}{\mathbbm e}
\newcommand{\h}{\mathbbm h}
\newcommand{\C}{\mathbbm c}
\newcommand{\0}{\mathbf 0}

\newcommand*{\ip}[1]{\left\langle{#1}\right\rangle} 
\newcommand*{\norm}[1]{\left\lVert{#1}\right\rVert} 
\newcommand*{\conj}[1]{\overline{#1}} 

\newcommand{\imq}{\mathrm{Im}}
\newcommand{\req}{\mathrm{Re}}

\renewcommand{\r}{\mathbf r}
\renewcommand{\i}{\mathbf i}
\renewcommand{\j}{\mathbf j}
\renewcommand{\k}{\mathbf k}

\DeclareMathOperator*{\argmin}{arg\,min}

\newcommand{\supp}{\mathrm {supp}}

\begin{document}

\begin{center}
\Large
\textbf{Compressed sensing for real measurements of quaternion signals}\\
\end{center}


\begin{center}
\begin{tabular}{cc}
	\textbf{Agnieszka Bade\'nska} & \textbf{{\L}ukasz B{\l}aszczyk} \\
	Faculty of Mathematics  & Institute of Radioelectronics \\ and Information Science & and Multimedia Technology \\
	Warsaw University of Technology & Warsaw University of Technology \\
	ul. Koszykowa 75 & ul. Nowowiejska 15/19 \\
	00-662 Warszawa & 00-665 Warszawa\\
	Poland & Poland \\
	badenska@mini.pw.edu.pl & l.blaszczyk@ire.pw.edu.pl
\end{tabular}
\end{center}


\bigskip
\normalsize
\begin{abstract}
The~article concerns compressed sensing methods in the~quaternion algebra. We prove that it is possible to uniquely reconstruct -- by $\ell_1$ norm minimization -- a~sparse quaternion signal from a~limited number of its real linear measurements, provided the~measurement matrix satisfies so-called restricted isometry property with a~sufficiently small constant. We also provide error estimates for the~reconstruction of a~non-sparse quaternion signal in the~noisy and noiseless cases.
\end{abstract}

\section{Introduction}

The~idea of compressed sensing is to recover a~sparse (supported on a~set of small cardinality) finite dimensional signal~$\x$ from a~few number of its linear measurements $\y=\Fi\x$, by solving the~convex program of $\ell_1$ norm minimization:
$$ \min\norm{\z}_1 \quad\text{subject to}\quad \Fi\z=\y. $$
It is well known that the~exact recovery is possible if the~measurement matrix $\Fi$ satisfies a~condition known as the~restricted isometry property (Definition~\ref{RIP}), introduced in~\cite{ct}, with sufficiently small constant (see e.g. \cite{cRIP,crt,ct} and \cite{introCS} for more references). Moreover, even if the~original signal is not sparse but e.g. compressible (most of its entries close to zero), the~same minimization provides a~good sparse approximation of the~signal and the~procedure is stable in the~sense that the~error is bounded above by the~$\ell_1$~norm of the~difference between the~original signal and its best sparse approximation.

More general, one can assume that the~observables are contaminated by a~white noise,
$$ \y=\Fi\x+\e, \quad\textrm{where}\quad \norm{\e}_2\leq\eta. $$
The~exact recovery is obviously impossible, however, if the~signal $\x$ was sparse, we still are able to reconstruct it in a~stable manner, i.e. with an~error bounded in terms of~$\eta$. To do so one solves a~modified convex problem
$$ \min\norm{\z}_1 \quad\text{subject to}\quad \norm{\y-\Fi\z}_2\leq\eta. $$

So far the~attention of researchers in compressed sensing has mostly been focused on the~case of real and complex signals and measurements. Our aim is to investigate if the~compressed sensing methods can be successfully applied also to quaternion signals. This generalization would be significant because of broad applications of the~quaternion algebra. Apart from classical applications, e.g. in quantum mechanics, for the~description of 3D solid body rotations, etc., quaternions have also been used in the~field of signal processing. Their structure is suitable for description of a~colour image -- the~imaginary part is interpreted in terms of three components of a~colour image: red, green and blue. That is why quaternions have found numerous applications in image filtering, pattern recognition, edge detection and watermarking \cite{G1,P1,T1,W1}. There has also been proposed a~dual-tree quaternion wavelet transform in a~multiscale analysis of geometric image features~\cite{ccb}. For this purpose an~alternative representation of quaternions is used -- through its magnitude (norm) and three phase angles: two of them encode phase shifts while the~third contains image texture information. 

The~motivation for this work was article~\cite{l1minqs}, the~authors of which performed a~numerical experiment of a~successful recovery of sparse quaternion signals from a~limited amount of their random Gaussian quaternion measurements. There has also been proposed an~algorithm for solving the~$\ell_1$ minimization problem in the~algebra of quaternions (by using the~second-order cone programming). However, to the~authors' best knowledge, so far in the~literature there has been no proof for any compressed sensing methods in the~algebra of quaternions.

In this article we deal with the~case of quaternion signals and their linear measurements with real coefficients. The~main results, stated in Theorem~\ref{l1minthm} and Corollary~\ref{l1mincor}, confirm the~numerical experiments from~\cite{l1minqs} and provide estimates on the~error for the~problem of reconstruction of a~quaternion signal (not necessarily sparse) from noisy and noiseless data by minimization of the~$\ell_1$ quaternion norm -- under the~condition that the~real measurements matrix satisfies the~restricted isometry property with a~sufficiently small constant. This is a~starting point for further research, i.e. investigating the~case of quaternion measurements of quaternion signals, search for 'good' measurement quaternion matrices, etc. 

The article is organized as follows. In the~next section we recall basic properties of quaternions and provide two versions of polarization identity. Section~3 is devoted to the~restricted isometry property and its consequences -- we prove Lemma~\ref{RIPip} which is an~important tool in the~proof of our main results. In the~sections~4 and~5 we state and prove the~main results -- Theorem~\ref{l1minthm} and Corollary~\ref{l1mincor}. Finally, section~6 presents results of a~numerical experiment for the~considered case, i.e. reconstruction (by $\ell_1$ minimization) of sparse quaternion signals from their linear measurements with real coefficients and error estimation for non-sparse quaternion signals giving a~lower bound on the~constant~$C_0$ from Corollary~\ref{l1mincor}.

\section{Algebra of quaternions}

Denote by~$\H$ the~algebra of quaternions
$$ q=a+b\i+c\j+d\k, \quad \textrm{where}\quad a,b,c,d\in\R $$
endowed with the~standard norm
$$ |q|=\sqrt{q\conj{q}}=\sqrt{a^2+b^2+c^2+d^2}, $$
where $\conj{q}=a-b\i-c\j-d\k$ is the~conjugate of~$q$. The~real part $\req(q)$ of $q=a+b\i+c\j+d\k$ is the~real number~$a$ while quaternion $b\i+c\j+d\k$ is called the~imaginary party of~$q$ and denoted by $\imq(q)$. The~conjugate of~$q$ can also be expressed as
$$ \conj{q}=-\frac{1}{2}\left(q+\i{q}\i+\j{q}\j+\k{q}\k\right). $$
Recall that multiplication is in general not commutative in the~quaternion algebra and is defined by the~following rules
$$ \i^2=\j^2=\k^2=\i\j\k=-1 $$
and
$$ \i\j=-\j\i=\k, \quad \j\k=-\k\j=\i, \quad \k\i=-\i\k=\j. $$
However, we have the~property that
$$ \conj{q\cdot w}=\conj{w}\cdot\conj{q} \quad\textrm{for any}\quad q,w\in\H. $$

For any $n\in\N$ we introduce the~following function $\ip{\cdot,\cdot}\colon\H^n\times\H^n\to\H$ with quaternion values:
$$ \ip{\x,\y}=\sum_{i=1}^{n}x_i\conj{y_i}, \quad\textrm{where}\quad \x=(x_1,\ldots,x_n)^T,\; \y=(y_1,\ldots,y_n)^T\in\H^n $$
and $T$ is the~transpose. Denote also
$$ \norm{\x}_2=\sqrt{\ip{\x,\x}}=\sqrt{\sum_{i=1}^{n}|x_i|^2}, \quad \textrm{for any }\x=(x_1,\ldots,x_n)^T\in\H^n. $$
As a~direct consequence of the~following lemma, $\norm{\cdot}_2$ is a~norm in~$\H^n$.

\begin{lemma}\label{ipproperties} The~function $\ip{\cdot,\cdot}$ satisfies axioms of the~inner product.
\end{lemma}
\begin{proof}
	Let $\x=(x_1,\ldots,x_n)^T,\y=(y_1,\ldots,y_n)^T,\z=(z_1,\ldots,z_n)^T\in\H^n$ and $\lambda\in\H$.
	\begin{itemize}
		\item $\displaystyle \conj{\ip{\x,\y}}=\conj{\sum_{i=1}^{n}x_i\conj{y_i}}=\sum_{i=1}^{n}\conj{x_i\conj{y_i}}=\sum_{i=1}^{n}y_i\conj{x_i}=\ip{\y,\x}$ .
		\item $\displaystyle \ip{\lambda\x,\y}=\sum_{i=1}^{n}\lambda x_i\conj{y_i}=\lambda\sum_{i=1}^{n}x_i\conj{y_i}=\lambda\ip{\x,\y}$.
		\item $\displaystyle \ip{\x+\y,\z}=\sum_{i=1}^{n}(x_i+y_i)\conj{z_i}=\sum_{i=1}^{n}x_i\conj{z_i}+\sum_{i=1}^{n}y_i\conj{z_i}=\ip{\x,\z}+\ip{\y,\z}$.
		\item $\displaystyle \ip{\x,\x}=\sum_{i=1}^{n}x_i\conj{x_i}=\sum_{i=1}^{n}|x_i|^2=\norm{\x}_2^2\geq0$.
		\item $\displaystyle \ip{\x,\x}=\norm{\x}_2^2=0 \quad \iff \quad \x=0$.
	\end{itemize}
\end{proof}

We have also the~Cauchy-Schwarz inequality.
\begin{lemma}\label{qCSineq}
	For any $\x,\y\in\H^n$,
	$$ |\ip{\x,\y}|\leq\norm{\x}_2\cdot\norm{\y}_2. $$
\end{lemma}
\begin{proof}
	We carefully follow the~classical steps of the~proof, using the~above properties and keeping the~order of terms in multiplication. Take any $\x,\y\in\H^n$ and $q\in\H$. If $\y=0$ we are done, hence assume that $\norm{\y}_2>0$. By Lemma~\ref{ipproperties} we have that
	$$\aligned 0\leq\ip{\x-q\y,\x-q\y}&=\ip{\x,\x}-q\ip{\y,\x}-\ip{\x,\y}\conj{q}+q\ip{\y,\y}\conj{q}\\
&=\norm{x}_2^2-q\conj{\ip{\x,\y}}-\ip{\x,\y}\conj{q}+|q|^2\norm{\y}_2^2. \endaligned$$
	Putting $q=\frac{\ip{\x,\y}}{\norm{y}^2}$ we get
	$$ 0\leq\norm{x}_2^2-\frac{\ip{\x,\y}\conj{\ip{\x,\y}}}{\norm{y}_2^2}-\frac{\ip{\x,\y}\conj{\ip{\x,\y}}}{\norm{y}_2^2}+\frac{|\ip{\x,\y}|^2}{\norm{y}_2^2}= 
			\norm{x}_2^2-\frac{|\ip{\x,\y}|^2}{\norm{y}_2^2}, $$
	which gives the result.
\end{proof}

The~function $\ip{\cdot,\cdot}$ is not a~standard inner product since its values are quaternions. However, we are able to obtain for it the~following versions of polarization identity.

\begin{theorem}[Polarization identity I]\label{polar1} For any $\x,\y\in\H^n$ we have
$$ \aligned \ip{\x,\y} &= \frac{1}{4}\left(\norm{\x+\y}_2^2-\norm{\x-\y}_2^2\right)\\	
											 &+\frac{\i}{4}\left(\norm{\x+\i\y}_2^2-\norm{\x-\i\y}_2^2\right)	\\
											 &+\frac{\j}{4}\left(\norm{\x+\j\y}_2^2-\norm{\x-\j\y}_2^2\right)	\\
											 &+\frac{\k}{4}\left(\norm{\x+\k\y}_2^2-\norm{\x-\k\y}_2^2\right). \endaligned $$
\end{theorem}
\begin{proof}
Denote $\x=(x_1,\ldots,x_n)^T$, $\y=(y_1,\ldots,y_n)^T$ and let us begin with the~real part.
$$\aligned \norm{\x+\y}^2-\norm{\x-\y}^2&=\sum_{i=1}^{n}\big((x_i+y_i)(\conj{x_i}+\conj{y_i})-(x_i-y_i)(\conj{x_i}-\conj{y_i})\big)\\
																		&=\sum_{i=1}^{n}\big(x_i\conj{x_i}+x_i\conj{y_i}+y_i\conj{x_i}+y_i\conj{y_i}-x_i\conj{x_i}+x_i\conj{y_i}+y_i\conj{x_i}-y_i\conj{y_i}\big)\\
																		&=2\sum_{i=1}^{n}\big(x_i\conj{y_i}+y_i\conj{x_i}\big)\\
																		&=2\big(\ip{\x,\y}+\ip{\y,\x}\big)=2\left(\ip{\x,\y}+\conj{\ip{\x,\y}}\right)=4\,\req\left(\ip{\x,\y}\right). \endaligned $$

Now, the~term with the~imaginary unit $\i$.
$$\aligned \norm{\x+\i\y}^2-\norm{\x-\i\y}^2&=\sum_{i=1}^{n}\left((x_i+\i y_i)\left(\conj{x_i}+\conj{\i y_i}\right)-(x_i-\i y_i)\left(\conj{x_i}-\conj{\i y_i}\right)\right)\\
																				&=\sum_{i=1}^{n}\left(x_i\conj{x_i}+x_i\conj{\i y_i}+\i y_i\conj{x_i}+\i y_i\conj{\i y_i}-x_i\conj{x_i}+x_i\conj{\i y_i}+\i y_i\conj{x_i}-\i y_i\conj{\i y_i}\right)\\
																				&=2\sum_{i=1}^{n}\big(-x_i\conj{y_i}\i+\i y_i\conj{x_i}\big)\\
																			&=-2\big(\ip{\x,\y}\i-\i\ip{\y,\x}\big)=-2\big(\ip{\x,\y}\i-\i\conj{\ip{\x,\y}}\big). \endaligned $$
Analogously for the~remaining imaginary units:
$$ \norm{\x+\j\y}_2^2-\norm{\x-\j\y}_2^2=-2\big(\ip{\x,\y}\j-\j\conj{\ip{\x,\y}}\big) $$
and
$$ \norm{\x+\k\y}_2^2-\norm{\x-\k\y}_2^2=-2\big(\ip{\x,\y}\k-\k\conj{\ip{\x,\y}}\big). $$
Multiplying the~left hand sides by $\i,\j,\k$ respectively and summing up the~identities we obtain that
$$\aligned \i\left(\norm{\x+\i\y}_2^2-\norm{\x-\i\y}_2^2\right)&+\j\left(\norm{\x+\j\y}_2^2-\norm{\x-\j\y}_2^2\right)+\k\left(\norm{\x+\k\y}_2^2-\norm{\x-\k\y}_2^2\right) \\
						&=-2\left(\i\ip{\x,\y}\i+\j\ip{\x,\y}\j+\k\ip{\x,\y}\k\right)-6\conj{\ip{\x,\y}}\\
						&=-2\left(-2\conj{\ip{\x,\y}}-\ip{\x,\y}\right)-6\conj{\ip{\x,\y}}\\
						&=2\ip{\x,\y}-2\conj{\ip{\x,\y}}=4\,\imq\left(\ip{\x,\y}\right), \endaligned $$
which finishes the proof.
\end{proof}

In order to formulate the~second version of the~last result, let us introduce a~different representation of a~quaternion $q\in\H$:
$$ q=x+uy, \quad \textrm{where}\quad x,y\in\R,\; y\geq0, \; u\in\H, \;\req(u)=0, \; |u|=1. $$
Then obviously
$$ x=\req(q), \quad uy=\imq(q)\quad \textrm{and}\quad |q|^2=x^2+y^2. $$
Since $\req(u)=0$ and $|u|=1$, we also have
$$ \conj{u}=-u \quad\textrm{and}\quad 1=u\,\conj{u}=u\,(-u)=-u^2. $$
Any quaternion with these properties can also be called an~imaginary unit (cf.~\cite{vc}).

\begin{lemma}\label{urepresentation} For any $q\in\H$ with $q=x+uy$, where $x,y\in\R$, $y\geq0$, $u\in\H$, $\req(u)=0$, $|u|=1$, we have
$$ \conj{u}\,q\,u=q \quad \textrm{and}\quad \conj{u}\,\conj{q}\,u=\conj{q}. $$
\end{lemma}
\begin{proof} Using the~fact that multiplying quaternions by real numbers is commutative, we get that
$$ \conj{u}\,q\,u=\conj{u}(x+uy)u=\conj{u}\,x\,u+\conj{u}\,u\,y\,u=x|u|^2+|u|^2y\,u=x+uy=q. $$
And the~second identity analogously.
\end{proof}

\begin{theorem}[Polarization identity II]\label{polar2} For any $\x,\y\in\H^n$, if we denote $\ip{\x,\y}=a+ub$, where $a,b\in\R$, $b\geq0$, $u\in\H$, $\req(u)=0$, $|u|=1$, we have
$$ \ip{\x,\y}=\frac{1}{4}\left(\norm{\x+\y}_2^2-\norm{\x-\y}_2^2\right)+\frac{u}{4}\left(\norm{\x+u\y}_2^2-\norm{\x-u\y}_2^2\right). $$
\end{theorem}
\begin{proof}
The~form of the~real part was established in the~previous theorem. Using the~fact that $\conj{u}=-u$ we get that
$$ \norm{\x+u\y}_2^2=\norm{\x}_2^2+u\conj{\ip{\x,\y}}-\ip{\x,\y}u+\norm{\y}_2^2 $$
and
$$ \norm{\x-u\y}_2^2=\norm{\x}_2^2-u\conj{\ip{\x,\y}}+\ip{\x,\y}u+\norm{\y}_2^2. $$
Hence, by Lemma~\ref{urepresentation},
$$\aligned u\left(\norm{\x+u\y}_2^2-\norm{\x-u\y}_2^2\right)&=2u^2\,\conj{\ip{\x,\y}}-2u\ip{\x,\y}u=-2\conj{\ip{\x,\y}}+2\,\conj{u}\ip{\x,\y}u\\
																												&=-2\conj{\ip{\x,\y}}+2\ip{\x,\y}=2\left(\ip{\x,\y}-\conj{\ip{\x,\y}}\right)\\
																												&=4\,\imq\left(\conj{\ip{\x,\y}}\right). \endaligned $$
\end{proof}

In what follows we will consider $\norm{\cdot}_p$ norms for quaternion vectors $\x\in\H^n$ defined in the~standard way:
$$ \norm{\x}_{p}=\left(\sum_{i=1}^{n}|x_i|^p\right)^{1/p}, \quad \textrm{for}\quad p\in[1,\infty) $$
and
$$ \norm{\x}_{\infty}=\max_{1\leq i\leq n}|x_i|, $$
where $\x=(x_1,\ldots,x_n)^T$.
We will also apply the~usual notation for the~cardinality of the~support of~$\x$, i.e.
$$ \norm{\x}_0=\#\supp(\x), \quad\textrm{where}\quad \supp(\x)=\{i\in\{1,\ldots,n\}: x_i\neq0\}. $$

\section{Restricted Isometry Property}

Recall that we call a~vector (signal) $\x\in\H^n$ $s$-sparse if it has at most $s$ nonzero coefficients, i.e.
$$ \norm{\x}_0\leq s. $$
As it was mentioned in the~introduction, one of the~conditions which guarantees exact reconstruction of a~sparse real signal from a~few number of its linear measurements is that the~measurement matrix satisfies so-called restricted isometry property (RIP) with a~sufficiently small constant. The~notion of restricted isometry constants was introduced by Cand\`es and Tao in~\cite{ct}.

\begin{definition}\label{RIP}
Let $s\in\N$ and $\Fi\in\R^{m\times n}$. We say that $\Fi$  satisfies the $s$-restricted isometry property (for real vectors) with a~constant $\delta_s\geq0$ if
$$ \left(1-\delta_s\right)\norm{\x}_2^2\leq\norm{\Fi\x}_2^2\leq\left(1+\delta_s\right)\norm{\x}_2^2 $$
for all s-sparse real vectors $\x\in\R^n$. The~smallest number $\delta_s\geq0$ with this property is called the $s$-restricted isometry constant. 
\end{definition}

Note that we can define $s$-isometry constants for any matrix $\Fi\in\R^{m\times n}$ and any number $s\in\{1,\ldots,n\}$. There are various examples of real matrices satisfying RIP (with overwhelming probability), e.g. Bernoulli/Gaussian random matrices with i.i.d. entries or, more general, random sampling matrices associated to bounded orthonormal systems  (cf.~\cite{crt,ct,introCS}). The~following observation allows us to use these matrices also for quaternion signals.

\begin{lemma}
If a~matrix $\Fi\in\R^{m\times n}$  satisfies the $s$-restricted isometry property (for real vectors) with a~constant $\delta_s\geq0$, then it satisfies the $s$-restricted isometry property for quaternion vectors with the~same constant, i.e.
$$ \left(1-\delta_s\right)\norm{\x}_2^2\leq\norm{\Fi\x}_2^2\leq\left(1+\delta_s\right)\norm{\x}_2^2 $$
for all s-sparse quaternion vectors $\x\in\H^n$.
\end{lemma}
\begin{proof}
Take any $s$-sparse vector $\x\in\H^n$ and express it as
$$ \x=\x_\r+\i\x_\i+\j\x_\j+\k\x_\k, \quad\textrm{where}\quad \x_i\in\R^n, \; i\in\{\r,\i,\j,\k\}. $$
Notice that all vectors $\x_i\in\R^n$ are also $s$-sparse. Moreover,
$$\Fi\x=\Fi\x_\r+\i\Fi\x_\i+\j\Fi\x_\j+\k\Fi\x_\k,\quad  \textrm{where}\quad \Fi\x_i\in\R^n, \; i\in\{\r,\i,\j,\k\}. $$
Hence 
$$ \norm{\Fi\x}_2^2=\norm{\Fi\x_\r}_2^2+\norm{\Fi\x_\i}_2^2+\norm{\Fi\x_\j}_2^2+\norm{\Fi\x_\k}_2^2 $$
and obviously
$$ \norm{\x}_2^2=\norm{\x_\r}_2^2+\norm{\x_\i}_2^2+\norm{\x_\j}_2^2+\norm{\x_\k}_2^2. $$
And the~result follows after applying the $s$-restricted isometry inequalities to each vector $\x_i\in\R^n$ separately.
\end{proof}

The~next result is an~important tool in the~proof of Theorem~\ref{l1minthm}. Note that for quaternion vectors we are not able to obtain the~same estimate as in the~real case. However, the~enhanced version of the~polarization identity (Theorem~\ref{polar2}) -- which is the~key ingredient in the~proof -- allows us to decrease the~multiplicative constant from $2$ to $\sqrt2$.

\begin{lemma}\label{RIPip}
 Let $\delta_s$ be the~$s$-isometry constant for a~matrix $\Fi\in\R^{m\times n}$ for $s\in\{1,\ldots,n\}$. For any pair of $\x,\y\in\H^n$ with disjoint supports and such that $\norm{\x}_0\leq s_1$ and $\norm{\y}_0\leq s_2$, where $s_1+s_2\leq n$, we have that 
$$  \left|\ip{\Phi\x,\Phi\y}\right|\leq \sqrt2 \delta_{s_1+s_2}\norm{\x}_2\norm{\y}_2. $$
\end{lemma}
\begin{proof}
First take two unit vectors $\x,\y\in\H^n$ satisfying the~assumptions of the~lemma. Since $\x$ and $\y$ have disjoint supports and $\norm{\x}_2=\norm{\y}_2=1$ we have that
$$ \norm{\x\pm\y}_2^2=\norm{\x}_2^2+\norm{\y}_2^2=2. $$
Moreover, for any quaternion $q\in\H$ with $|q|=1$,
$$ \norm{\x\pm q\y}_2^2=\norm{\x}_2^2+|q|^2\norm{\y}_2^2=2. $$

Denote $\ip{\Fi\x,\Fi\y}=a+ub$, where $a,b\in\R$, $b\geq0$, $u\in\H$, $\req(u)=0$, $|u|=1$. Applying the~polarization identity from Theorem~\ref{polar2} we get that
$$ \ip{\Fi\x,\Fi\y}=\frac{1}{4}\left(\norm{\Fi\x+\Fi\y}_2^2-\norm{\Fi\x-\Fi\y}_2^2\right)+\frac{u}{4}\left(\norm{\Fi\x+u\Fi\y}_2^2-\norm{\Fi\x-u\Fi\y}_2^2\right). $$
Since $\x\pm\y$ are $(s_1+s_2)$-sparse, 
$$ \left(1-\delta_{s_1+s_2}\right)\underbrace{\norm{\x\pm\y}_2^2}_{=2}\leq\norm{\Fi\x\pm\Fi\y}_2^2=\norm{\Fi(\x\pm\y)}_2^2\leq\left(1+\delta_{s_1+s_2}\right)\underbrace{\norm{\x\pm\y}_2^2}_{=2} $$
and therefore
$$ \norm{\Fi\x+\Fi\y}_2^2-\norm{\Fi\x-\Fi\y}_2^2\leq2\left(1+\delta_{s_1+s_2}\right)-2\left(1-\delta_{s_1+s_2}\right)=4\delta_{s_1+s_2}. $$

Similarly, since $\Fi$ is real and $\x\pm u\y$ are also $(s_1+s_2)$-sparse,
$$ \left(1-\delta_{s_1+s_2}\right)\underbrace{\norm{\x\pm u\y}_2^2}_{=2}\leq\norm{\Fi\x\pm u\Fi\y}_2^2 =\norm{\Fi(\x\pm u\y)}_2^2\leq\left(1+\delta_{s_1+s_2}\right)\underbrace{\norm{\x\pm u\y}_2^2}_{=2}, $$
hence 
$$ \norm{\Fi\x+u\Fi\y}_2^2-\norm{\Fi\x-u\Fi\y}_2^2\leq2\left(1+\delta_{s_1+s_2}\right)-2\left(1-\delta_{s_1+s_2}\right)=4\delta_{s_1+s_2}. $$

Finally
$$ \left|\ip{\x,\y}\right|\leq\frac{1}{4}\sqrt{4^2\delta_{s_1+s_2}^2+4^2\delta_{s_1+s_2}^2}=\sqrt2\,\delta_{s_1+s_2}. $$

Now, if $\x,\y\in\H^n$ are any vectors satisfying assumptions of the~lemma, applying the~above estimate we conclude that
$$ \left|\ip{\x,\y}\right|=\norm{\x}_2\norm{\y}_2\left|\ip{\frac{\x}{\norm{\x}_2},\frac{\y}{\norm{\y}_2}}\right|\leq\sqrt2\,\delta_{s_1+s_2}\norm{\x}_2\norm{\y}_2.  $$
\end{proof}

\section{Stable reconstruction from noisy data}

As we mentioned in the~introduction, our aim is to reconstruct a~quaternion signal from a~limited amount of its linear measurements with real coefficients. We will also assume the~presence of a~white noise with bounded $\ell_2$ quaternion norm. The~observables are, therefore, given by
$$ \y=\Fi\x+\e, \quad\textrm{where}\quad \x\in\H^n, \; \Fi\in\R^{m\times{n}}, \; \y\in\H^m \;\textrm{and}\; \e\in\H^n \textrm{ with } \norm{\e}_2\leq\eta $$
for some $m\leq n$ and $\eta\geq0$.

We will use the~following notation: for any $\h\in\H^n$ and a~set of indices $T\subset\{1,\ldots,n\}$, the~vector $\h_T\in\H^n$ is supported on~$T$ with entries
$$ \left(\h_T\right)_i=\left\{\begin{array}{cl} h_i & \textrm{if } i\in T \\ 0 & \textrm{otherwise} \end{array}\right., \quad \textrm{where} \quad \h=(h_1,\ldots,h_n)^T. $$
The~complement of $T\subset\{1,\ldots,n\}$ will be denoted by $T^c=\{1,\ldots,n\}\setminus{T}$ and the~symbol $\x|_s$ will be used for the~best $s$-sparse approximation of the~vector~$\x$, i.e. $\x|_s=\x_{T_0}$, where $T_0$ is the set of indices of $\x$~coordinates with the biggest quaternion norms.

The~following result is a~generalization of {\cite[Theorem~1.3]{cRIP}} to the~case of quaternion signals.

\begin{theorem}\label{l1minthm}
Suppose that $\Fi\in\R^{m\times n}$ satisfies the~$2s$-restricted isometry property with a~constant $\delta_{2s}<\frac{1}{3}$ and let $\eta\geq0$. Then, for any $\x\in\H^n$ and $\y=\Fi\x+\e$ with $\norm{\e}_2\leq\eta$, the~solution $\x^{\#}$ of the~problem
\begin{equation}\label{l1minproblem}
\argmin\limits_{\z\in\H^n} \norm{\z}_1 \quad\text{subject to}\quad\norm{\Fi\z-\y}_2\leq\eta
\end{equation}
satisfies 
\begin{equation} \label{main-ineq}
\norm{\x^{\#}-\x}_2 \leq \frac{C_0}{\sqrt{s}}\norm{\x-\x|_s}_1 + C_1\eta
\end{equation} with constants 
$$ C_0 = \frac{2(1+\delta_{2s})}{1-3\delta_{2s}} ,\quad C_1 = \frac{4\sqrt{1+\delta_{2s}}}{1-3\delta_{2s}}, $$
where $\x|_s$ denotes the~best $s$-sparse approximation of $\x$.
\end{theorem}

\begin{proof}
First, note that, since $\x^{\#}$ is the~minimizer of~(\ref{l1minproblem}) and $\x$ is feasible, we get that
\begin{equation}\label{Fidifference}
	\norm{\Fi\left(\x^{\#}-\x\right)}_2\leq  \norm{\Fi\x^{\#}-\y}_2+\norm{\Fi\x-\y}_2\leq 2\eta.
\end{equation}
Denote 
$$ \h=\x^{\#}-\x $$
and decompose $\h$ into a~sum of vectors $\h_{T_0},\h_{T_1},\h_{T_2}\ldots$ in the~following way: let $T_0$ be the~set of indices of $\x$ coordinates with the~biggest quaternion norms; $T_1$ is the~set of indices of $\h_{T_0^c}$ coordinates with the~biggest norms, $T_1$ is the~set of indices of of $\h_{\left(T_0\cup T_1\right)^c}$ coordinates with the~biggest norms, etc. Then obviously $\h_{T_j}$ are $s$-sparse and have disjoint supports.

Notice that for $j\geq2$ we have that
$$ \norm{\h_{T_j}}_2^2=\sum_{i\in T_j}|h_i|^2\leq \sum_{i\in T_j}\norm{\h_{T_j}}_{\infty}^2=s\norm{\h_{T_j}}_{\infty}^2, $$
where $h_i$ are the~coordinates of~$\h$ and $\norm{\h_{T_j}}_{\infty}=\max\limits_{i\in{T_j}}|h_i|$.
Moreover, since all non-zero coordinates of~$\h_{T_{j-1}}$ have norms not smaller than non-zero coordinates of $\h_{T_j}$,
$$ \norm{\h_{T_j}}_{\infty}\leq\frac{1}{s}\sum_{i\in{T_{j-1}}}|h_i|=\frac{1}{s}\norm{\h_{T_{j-1}}}_1. $$
Hence, for $j\geq2$ we get that
$$ 	\norm{\h_{T_j}}_2\leq \sqrt{s}\norm{\h_{T_j}}_{\infty}\leq \frac{1}{\sqrt{s}}\norm{\h_{T_{j-1}}}_1, $$
which implies
\begin{equation}\label{sumnormhTj}
 \sum_{j\geq2}\norm{\h_{T_j}}_2\leq \frac{1}{\sqrt{s}}\sum_{j\geq1}\norm{\h_{T_j}}_1\leq \frac{1}{\sqrt{s}}\norm{\h_{T_0^c}}_1.
\end{equation}
Finally 
\begin{equation}\label{normhT01c}
	\norm{\h_{\left(T_0\cup T_1\right)^c}}_2=\norm{\sum_{j\geq2}\h_{T_j}}_2\leq \sum_{j\geq2}\norm{\h_{T_j}}_2\leq \frac{1}{\sqrt{s}}\norm{\h_{T_0^c}}_1.
\end{equation}

Observe that $\norm{\h_{T_0^c}}_1$ can not be to large since $\norm{\x^{\#}}_1=\norm{\x+\h}_1$ is minimal. Indeed,
$$ \norm{\x}_1\geq\norm{\x+\h}_1= \norm{\x_{T_0}+\h_{T_0}}_1+ \norm{\x_{T_0^c}+\h_{T_0^c}}_1 \geq \norm{\x_{T_0}}_1-\norm{\h_{T_0}}_1-\norm{\x_{T_0^c}}_1+\norm{\h_{T_0^c}}_1, $$
hence
$$ \norm{\x_{T_0^c}}_1=\norm{\x}_1-\norm{\x_{T_0}}_1\geq -\norm{\h_{T_0}}_1-\norm{\x_{T_0^c}}_1+\norm{\h_{T_0^c}}_1 $$
and therefore
\begin{equation}\label{normhT0c}
	\norm{\h_{T_0^c}}_1\leq \norm{\h_{T_0}}_1+2\norm{\x_{T_0^c}}_1.
\end{equation}

Now, the Cauchy-Schwarz inequality immediately implies that
$$ \norm{h_{T_0}}_1=\sum_{i\in T_0}|h_i|\cdot1\leq\sqrt{\sum_{i\in T_0}|h_i|^2}\cdot \sqrt{\sum_{i\in T_0}1^2}=\sqrt{s}\norm{\h_{T_0}}_2 $$
From this, (\ref{normhT01c}) and (\ref{normhT0c}) we conclude that
\begin{equation}\label{ineq1} 
	\norm{\h_{\left(T_0\cup T_1\right)^c}}_2\leq \frac{1}{\sqrt{s}}\norm{\h_{T_0^c}}_1 \leq \frac{1}{\sqrt{s}}\norm{\h_{T_0}}_1+\frac{2}{\sqrt{s}}\norm{\x_{T_0^c}}_1\leq 
	\norm{\h_{T_0}}_2+2\epsilon,
\end{equation}
where $\epsilon=\frac{1}{\sqrt{s}}\norm{\x_{T_0^c}}_1=\frac{1}{\sqrt{s}}\norm{\x-\x|_s}_1$. This is the~first ingredient of the~final estimate.

\medskip
In what follows we are going to estimate the~remaining component, i.e. $\norm{\h_{\left(T_0\cup T_1\right)^c}}_2$. Since
$$ \Fi\h_{T_0\cup T_1}= \Fi\left(\h-\sum_{j\geq2}\h_{T_j}\right)=\Fi\h-\sum_{j\geq2}\Fi\h_{T_j}, $$
using linearity of $\ip{\cdot,\cdot}$, we get that
$$\aligned
	\norm{\Fi\h_{T_0\cup T_1}}_2^2&=\ip{\Fi\h_{T_0\cup T_1},\Fi\h_{T_0\cup T_1}}=\ip{\Fi\h_{T_0\cup T_1},\Fi\h}-\sum_{j\geq2}\ip{\Fi\h_{T_0\cup T_1},\Fi\h_{T_j}}\\
		&=\ip{\Fi\h_{T_0\cup T_1},\Fi\h}-\sum_{j\geq2}\ip{\Fi\h_{T_0},\Fi\h_{T_j}}-\sum_{j\geq2}\ip{\Fi\h_{T_1},\Fi\h_{T_j}}.
\endaligned$$

Estimate of the~first element follows from Lemma~\ref{qCSineq}, (\ref{Fidifference}) and RIP
\begin{equation}\label{FihT01Fih}
	\left|\ip{\Fi\h_{T_0\cup T_1},\Fi\h}\right|\leq\norm{\Fi\h_{T_0\cup T_1}}_2\cdot\norm{\Fi\h}_2 \leq \sqrt{1+\delta_{2s}}\norm{\h_{T_0\cup T_1}}_2\cdot2\eta.
\end{equation}
For the~remaining terms recall that $\h_{T_j}$ are $s$-sparse with disjoint supports and apply Lemma~\ref{RIPip}
$$ \left|\ip{\Fi\h_{T_0},\Fi\h_{T_j}}\right|\leq\sqrt2\,\delta_{2s}\cdot\norm{\h_{T_0}}_2\cdot\norm{\h_{T_j}}_2, \quad j\geq2, $$
$$ \left|\ip{\Fi\h_{T_1},\Fi\h_{T_j}}\right|\leq\sqrt2\,\delta_{2s}\cdot\norm{\h_{T_1}}_2\cdot\norm{\h_{T_j}}_2, \quad j\geq2. $$

Since $T_0$ and $T_1$ are disjoint, $\norm{\h_{T_0\cup T_1}}_2^2=\norm{\h_{T_0}}_2^2+\norm{\h_{T_1}}_2^2$ and therefore
$$ \norm{\h_{T_0}}_2+\norm{\h_{T_1}}_2\leq \sqrt2\norm{\h_{T_0\cup T_1}}_2 $$
since for any $a,b\in\R$ we have $(a+b)^2\leq2(a^2+b^2)$. Hence, using the~RIP, (\ref{FihT01Fih}) and (\ref{sumnormhTj}),
$$\aligned 
	\left(1-\delta_{2s}\right)\norm{\h_{T_0\cup T_1}}_2^2 &\leq \norm{\Fi\h_{T_0\cup T_1}}_2^2 \\
	&\leq \sqrt{1+\delta_{2s}}\norm{\h_{T_0\cup T_1}}_2\cdot2\eta+ \sqrt2\,\delta_{2s}\cdot\left(\norm{\h_{T_0}}_2+\norm{\h_{T_1}}_2\right)\sum_{j\geq2}\norm{\h_{T_j}}_2\\
	&\leq \left(2\sqrt{1+\delta_{2s}}\cdot\eta+\frac{2\delta_{2s}}{\sqrt{s}}\norm{\h_{T_0^c}}_1\right)\norm{\h_{T_0\cup T_1}}_2,
\endaligned$$
which implies that
\begin{equation}\label{normhT01}	
	\norm{\h_{T_0\cup T_1}}_2\leq \frac{2\sqrt{1+\delta_{2s}}}{1-\delta_{2s}}\cdot\eta+\frac{2\delta_{2s}}{1-\delta_{2s}}\cdot\frac{\norm{\h_{T_0^c}}_1}{\sqrt{s}}. 
\end{equation}
This, together with (\ref{normhT0c}), gives the~following estimate
$$ \norm{\h_{T_0\cup T_1}}_2\leq \alpha\cdot\eta+\beta\cdot\frac{\norm{\h_{T_0}}_1}{\sqrt{s}}+2\beta\cdot\frac{\norm{\x_{T_0^c}}_1}{\sqrt{s}}, $$
where
$$ \alpha=\frac{2\sqrt{1+\delta_{2s}}}{1-\delta_{2s}}\qquad \textrm{and}\qquad \beta=\frac{2\delta_{2s}}{1-\delta_{2s}}. $$
Since $\norm{\h_{T_0}}_1\leq\sqrt{s}\norm{\h_{T_0}}_2\leq\sqrt{s}\norm{\h_{T_0\cup T_1}}_2$, therefore
$$ \norm{\h_{T_0\cup T_1}}_2\leq \alpha\cdot\eta+\beta\cdot\norm{\h_{T_0\cup T_1}}_2+2\beta\cdot\epsilon,$$
where recall that $\epsilon=\frac{1}{\sqrt{s}}\norm{\h_{T_0^c}}_2$, hence
\begin{equation}\label{ineq2}
	\norm{\h_{T_0\cup T_1}}_2\leq\frac{1}{1-\beta}\left(\alpha\cdot\eta+2\beta\cdot\epsilon\right),
\end{equation}
as long as $\beta<1$ which is equivalent to $\delta_{2s}<\frac13$.

Finally, (\ref{ineq1}) and (\ref{ineq2}) imply the main result
$$\aligned
	\norm{h}_2&\leq\norm{\h_{T_0\cup T_1}}_2+\norm{\h_{(T_0\cup T_1)^c}}_2\leq \norm{\h_{T_0\cup T_1}}_2+\norm{\h_{T_0}}_2+2\epsilon \\
		&\leq 2\norm{\h_{T_0\cup T_1}}_2+2\epsilon\leq \frac{2\alpha}{1-\beta}\cdot\eta+\left(\frac{4\beta}{1-\beta}+2\right)\cdot\epsilon
\endaligned $$
and the constants in the~statement of the~theorem equal
$$ C_0=\frac{4\beta}{1-\beta}=2\frac{1+\beta}{1-\beta}=2\frac{1+\delta_{2s}}{1-3\delta_{2s}} \quad \textrm{and}\quad 
		C_1=\frac{2\alpha}{1-\beta}=\frac{4\sqrt{1+\delta_{2s}}}{1-3\delta_{2s}}.$$
\end{proof}

\section{Stable reconstruction from noiseless data}

In this section we will assume that our observables are exact, i.e.
$$ \y=\Fi\x, \quad\textrm{where}\quad \x\in\H^n, \; \Fi\in\R^{m\times{n}}, \; \y\in\H^m.$$
The~undermentioned result is a~natural corollary of~Theorem~\ref{l1minthm} for $\eta=0$.

\begin{corollary}\label{l1mincor}
	Let $\Fi\in\R^{m\times n}$ satisfies the~$2s$-restricted isometry property with a~constant $\delta_{2s}<\frac{1}{3}$. Then for any $\x\in\H^n$ and $\y=\Fi\x\in\H^m$, the~solution $\x^{\#}$ of the~problem
\begin{equation}\label{l1minexact}
\argmin\limits_{\z\in\H^n} \norm{\z}_1 \quad\text{subject to}\quad \Fi\z=\y
\end{equation}
satisfies 
\begin{equation} \label{exact-ineq1}
\norm{\x^{\#}-\x}_1 \leq C_0\norm{\x-\x|_s}_1 
\end{equation} 
\begin{equation} \label{exact-ineq2}
\norm{\x^{\#}-\x}_2 \leq \frac{C_0}{\sqrt{s}}\norm{\x-\x|_s}_1 
\end{equation} 
with constant $C_0$ as in the~Theorem~\ref{l1minthm}. In particular, if $\x$ is $s$-sparse and there is no noise, then the~reconstruction by $\ell_1$ minimization is exact.
\end{corollary}

\begin{proof}
	Inequality (\ref{exact-ineq2}) follows directly from Theorem~\ref{l1minthm} for $\eta=0$. The~result for sparse signals is obvious since in this case $\x=\x|_s$. We only need to prove~(\ref{exact-ineq1}).
	
	We will use the same notation as in the~proof of Theorem~\ref{l1minthm}. Recall that
	$$ \norm{\h_{T_0}}_1\leq\sqrt{s}\norm{\h_{T_0}}_2\leq\sqrt{s}\norm{\h_{T_0\cup T_1}}_2 $$
	which together with (\ref{normhT01}) for $\eta=0$ implies
$$ \norm{\h_{T_0}}_1\leq \frac{2\delta_{2s}}{1-\delta_{2s}}\cdot\norm{\h_{T_0^c}}_1  $$	
Using this and (\ref{normhT0c}), denoting again $\beta=\frac{2\delta_{2s}}{1-\delta_{2s}}$, we get that
$$ \norm{\h_{T_0^c}}_1\leq \beta\norm{\h_{T_0^c}}_1+2\norm{\x_{T_0^c}}_1, \qquad \textrm{hence}\qquad \norm{\h_{T_0^c}}_1\leq \frac{2}{1-\beta}\norm{\x_{T_0^c}}_1. $$
Finally, we obtain the~following estimate on the~$\ell_1$ norm of the~vector $\h=\x^{\#}-\x$
$$ \norm{\h}_1=\norm{\h_{T_0}}_1+\norm{\h_{T_0^c}}_1\leq (1+\beta)\norm{\h_{T_0^c}}_1\leq \underbrace{2\frac{1+\beta}{1-\beta}}_{=C_0}\norm{\x_{T_0^c}}_1, $$
which finishes the proof.
\end{proof}

We conjecture that the~ requirement $\delta_{2s}<\frac{1}{3}$ is not optimal -- there are known refinements of this condition for real signals (see e.g. {\cite[Chapter~6]{introCS}} for references). However, the~authors of~\cite{cz} constructed examples of $s$-sparse real signals which can not be uniquely reconstructed via $\ell_1$ minimization for $\delta_s>\frac13$. This gives an~obvious upper bound for~$\delta_{s}$ also for the~general quaternion case.

\section{Numerical experiment}

Inspired by the~article~\cite{l1minqs} we performed numerical experiments for the~case considered in this~paper, i.e. when the~measurement matrix $\Fi$ is real and satisfies (with overwhelming probability) the~restricted isometry property and signals are quaternion vectors. We applied the~algorithm described in~\cite{l1minqs} which express the~$\ell_1$ quaternion norm minimization problem in terms of the~second-order cone programming (SOCP). 

As it was also pointed out in~\cite{l1minqs}, our problem is equivalent to 
\begin{equation} \label{num:01}
	\argmin\limits_{t\in\R_+} t \quad\text{subject to}\quad \y = \Fi\x,\quad \norm{\x}_1\leq t.
\end{equation} 
We decompose $$ t=\sum\limits_{k=1}^n t_k,\; \textrm{ where }\; t_k\in\R_+, $$ 
and
$$	\x = \x_\r + \i\x_\i + \j\x_\j + \k\x_\k, \;\textrm{ where }\; \x_\r,\x_\i,\x_\j,\x_\k\in\R^n,$$ 
and denote
$$	\begin{array}{c} \x_\r = (x_{\r,1},x_{\r,2},\ldots,x_{\r,n})^T, \\
	\x_\i = (x_{\i,1},x_{\i,2},\ldots,x_{\i,n})^T, \\
	\x_\j = (x_{\j,1},x_{\j,2},\ldots,x_{\j,n})^T, \\
	\x_\k = (x_{\k,1},x_{\k,2},\ldots,x_{\k,n})^T. \end{array}$$
We also denote $\Fi = (\Ffi_1,\ldots,\Ffi_n)$, where $\Ffi_k\in\R^m$ for $k\in\{1,2,\ldots,n\}$. Then, we can write the second constraint from \eqref{num:01} as 
\begin{equation} \label{num:04}
	\norm{(x_{\r,k},x_{\i,k},x_{\j,k},x_{\k,k})^T}_2 \leq t_k\quad\textrm{for}\quad k\in\{1,2,\ldots,n\}.
\end{equation} 
This allows us to express our optimization problem \eqref{num:01} in the~real-vector setup in the~following way
\begin{equation}\aligned\label{num:05}
	\argmin\limits_{\tilde{\x}\in\R^n} \C^T\tilde{x} \quad\text{subject to}\quad &\tilde{\y} = \tilde{\Fi}\tilde{\x} \quad\\
\textrm{and}\quad &\norm{(x_{\r,k},x_{\i,k},x_{\j,k},x_{\k,k})^T}_2 \leq t_k,\quad k\in\{1,2,\ldots,n\},
\endaligned \end{equation} 
where 
\begin{equation}\aligned\label{num:06}
	\tilde{\x} &= (t_1,x_{\r,1},x_{\i,1},x_{\j,1},x_{\k,1},\ldots,t_n,x_{\r,n},x_{\i,n},x_{\j,n},x_{\k,n})^T\in\R^{5n},\\ 
	\C &= (1,0,0,0,0,\ldots,1,0,0,0,0)^T\in\R^{5n},\\ 
	\tilde{\y} &= (\y_\r^T,\y_\i^T,\y_\j^T,\y_\k^T)^T\in\R^{4m}, \\ 
	\tilde{\Fi} &= \left( \begin{array} {ccccccccccc}
		\0 & \Ffi_1 & \0 & \0 & \0 & \ldots & \0 & \Ffi_n & \0 & \0 & \0 \\
		\0 & \0 & \Ffi_1 & \0 & \0 & \ldots & \0 & \0 & \Ffi_n & \0 & \0 \\
		\0 & \0 & \0 & \Ffi_1 & \0 & \ldots & \0 & \0 & \0 & \Ffi_n & \0 \\
		\0 & \0 & \0 & \0 & \Ffi_1 & \ldots & \0 & \0 & \0 & \0 & \Ffi_n 
	\end{array}\right)\in\R^{4m\times 5n}. 
\endaligned \end{equation} 
This is a~standard form of the~SOCP and we solved it using the~Matlab toolbox SeDuMi~\cite{SeDuMi}. Finally, the~quaternion signal $\x^{\#}$, which is the~solution of~\eqref{num:01}, can easily be obtained from $\tilde{\x}$.

Our program was carried out on a standard PC machine, with Microsoft Windows 8.1 Pro system with Intel(R) Core(TM) i7-4790 CPU (3.60GHz) and 8GB RAM. The~following algorithm was implemented in Matlab R2011b: \begin{enumerate}
\item Fix constants $n=128$ (length of the signal $\x$) and $m$ (number of observables, i.e. size of the~vector~$\y$) and generate the measurement matrix $\Fi\in\R^{m\times n}$ with random entries from i.i.d. standard normal distribution $\mathcal{N}(0,1)$; 
\item Choose the~sparsity $s\leq\frac{n}{2}$, select a~support set $T$ ($\#T=s$) uniformly at random and generate a~vector $\x\in\H^n$ supported on $T$ with i.i.d. standard normal distribution (in the~quaternion $\ell_2$-norm sense with independent components); 
\item Compute observables $\y=\Fi\x$, $\y\in\H^m$;
\item Construct vectors $\tilde{\x}$, $\C$, $\tilde{\y}$ and matrix $\tilde{\Fi}$ as in \eqref{num:06}; 
\item Call the SeDuMi toolbox to solve SOCP problem formulated in \eqref{num:05} and calculate the~quaternion reconstructed vector $\x^{\#}$;
\item Compute the~error of reconstruction, i.e. $\norm{\x^{\#}-\x}_2$;
\item For each pair of $n$ and $s$ perform 100 experiments and save errors of each reconstruction and number of perfect reconstructions (we claim that the~reconstruction is perfect if $\norm{\x^{\#}-\x}_2\leq 10^{-7}$).
\end{enumerate} 

Next, we performed the same experiment for more general case, i.e. $\Fi\in\H^{m\times n}$, as in~\cite{l1minqs}. Fig.~\ref{fig:01}(a) shows how the~percentage of the~perfect recovery depends on the~number of measurements~$m$ and the~sparsity~$s$ for $n=128$. Fig. \ref{fig:01}(b) shows the same results for the general case $\Fi\in\H^{m\times n}$. We can see, for example, that for $s\leq\frac{m}{4}$ and $m=32$ the recovery rate is greater than $95\%$ in the~first experiment. Notice that this result is slightly worse than in the~general case. So far we do not know the~direct reason for this observation, however, it may be that the~random quaternion matrix has better properties than the~random real matrix (where three components are fixed to be zeros and only the~real part is chosen at random). We plan to further study this issue in future research.

\begin{figure}[h!]
\centering
	\begin{minipage}{7.5cm}
	\centering
 		\ifpdf
		  \includegraphics[width=7.5cm]{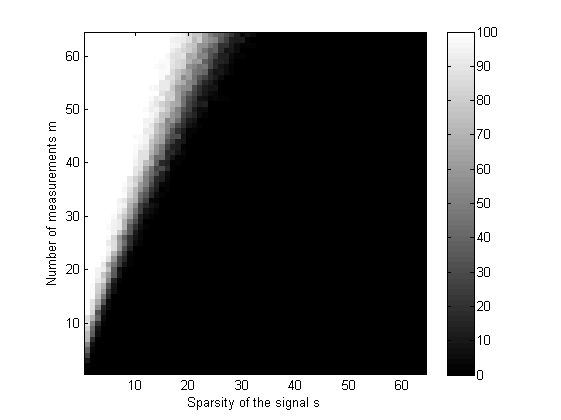}\\ (a)
		\else
		  \includegraphics[width=7.5cm]{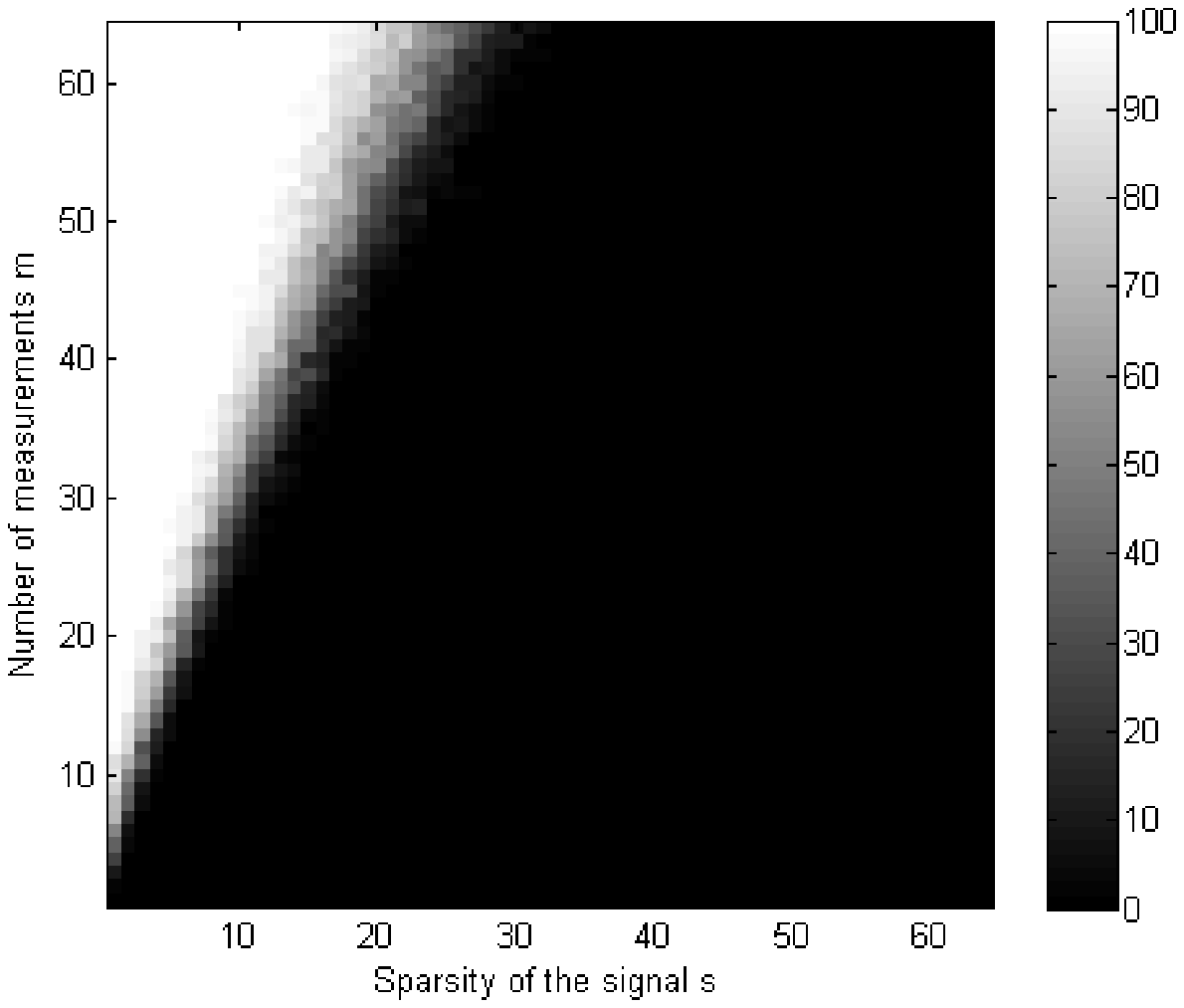}\\ (a)
		\fi
 	\end{minipage}
 	\begin{minipage}{7.5cm}
 	\centering
 		\ifpdf
		  \includegraphics[width=7.5cm]{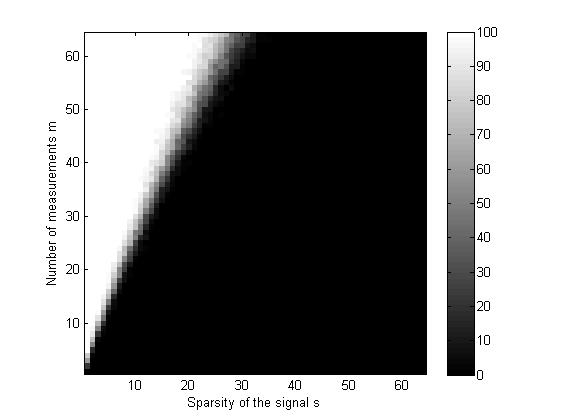}\\ (b)
		\else
		  \includegraphics[width=7.5cm]{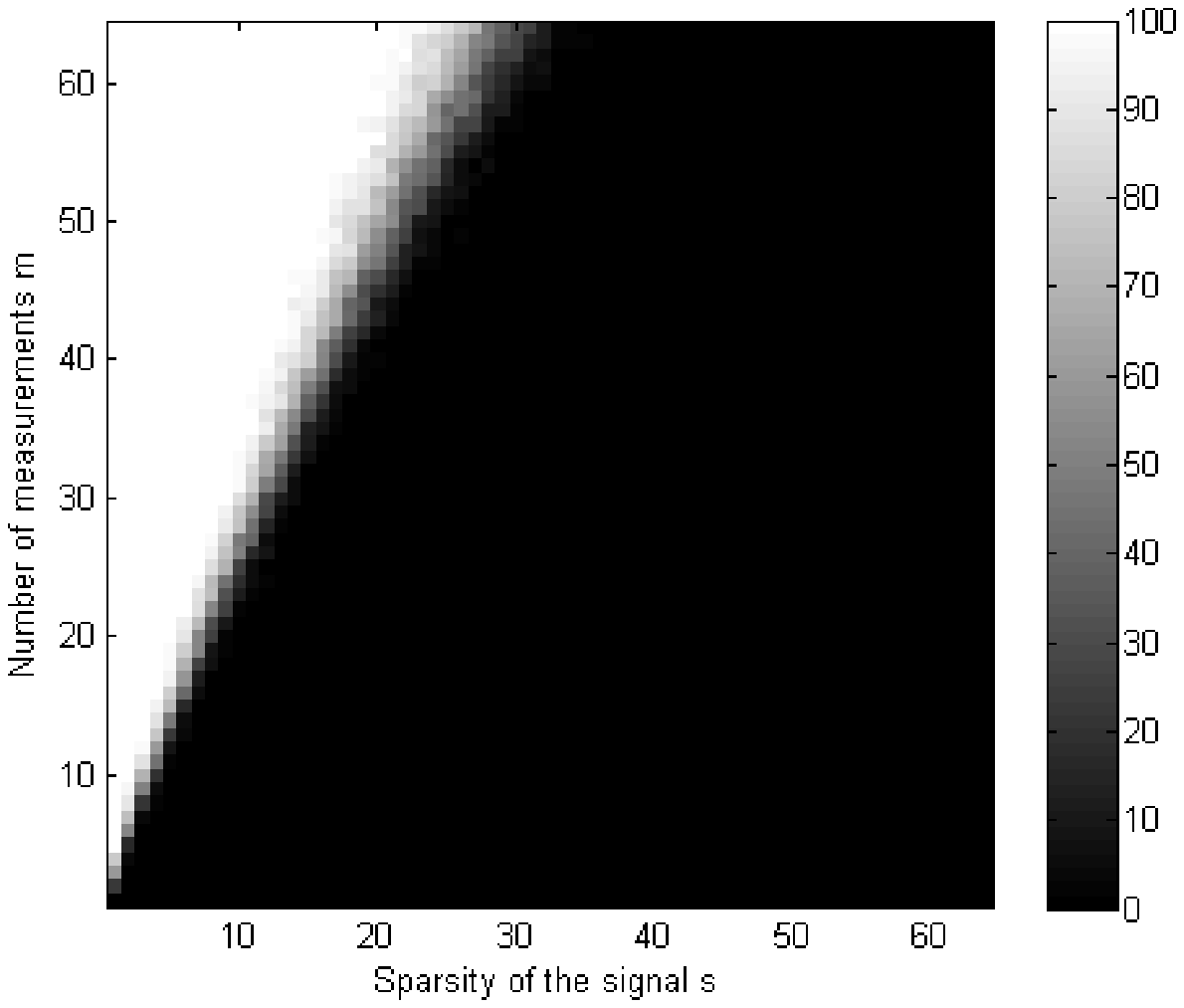}\\ (b)
		\fi
 	\end{minipage}
	\caption{\label{fig:01} Results of the recovery experiment for $n=128$. The image intensity stands for the percentage of perfect reconstructions. (a) The case for $\Fi\in\R^{m\times n}$. (b) The case for $\Fi\in\H^{m\times n}$.}
\end{figure}

To further illustrate the results formulated in Theorem~\ref{l1minthm} and Corollary~\ref{l1mincor} we performed another experiment. We fixed constants $n=256$ and $m=32$ and generated the measurement matrix $\Fi\in\R^{m\times n}$ with random entries from i.i.d. standard normal distribution and an~arbitrary vector $\x\in\H^n$ (not assuming sparsity) with random quaternion entries from i.i.d. standard normal distribution (in the~quaternion $\ell_2$-norm sense with independent components). We performed the~reconstruction of the vector $\x$, using the~algorithm described above, and calculated errors of reconstruction $\norm{\x^{\#}-\x}_1$ and $\norm{\x^{\#}-\x}_2$. We used the~inequalities \eqref{exact-ineq1} and \eqref{exact-ineq2} to obtain a~lower bound on the~constant $C_0$ as a~function of~$s$ (Fig.~\ref{fig:02}). 
We repeated the~experiment for various choices of $m$ and $n$ but the~results where comparable.

As we see the~results slightly differ (which is not surprising since we use \eqref{exact-ineq2} to prove~\eqref{exact-ineq1}), but we still obtain a~lower bound for $C_0$ as the~maximum of those two values (for arbitrary $s$). Observe that, as expected from the~statement of Theorem~\ref{l1minthm} and Corollary~\ref{l1mincor}, the~dependence on~$s$ is monotone. Note also that the~empirical bound on~$C_0$ is smaller than~$2$ for all~$s$, whereas the~theoretical formula gives $C_0>2$. We suspect, therefore, that our result is not sharp and can be improved.

\begin{figure}[h!]
\centering
	\begin{minipage}{7.5cm}
	\centering
 		\ifpdf
		  \includegraphics[width=7.5cm]{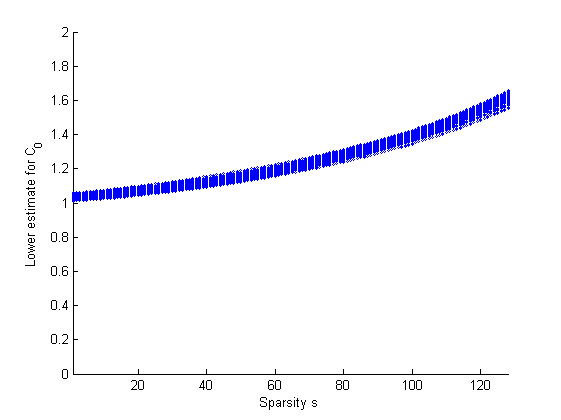}\\ (a)
		\else
		  \includegraphics[width=7.5cm]{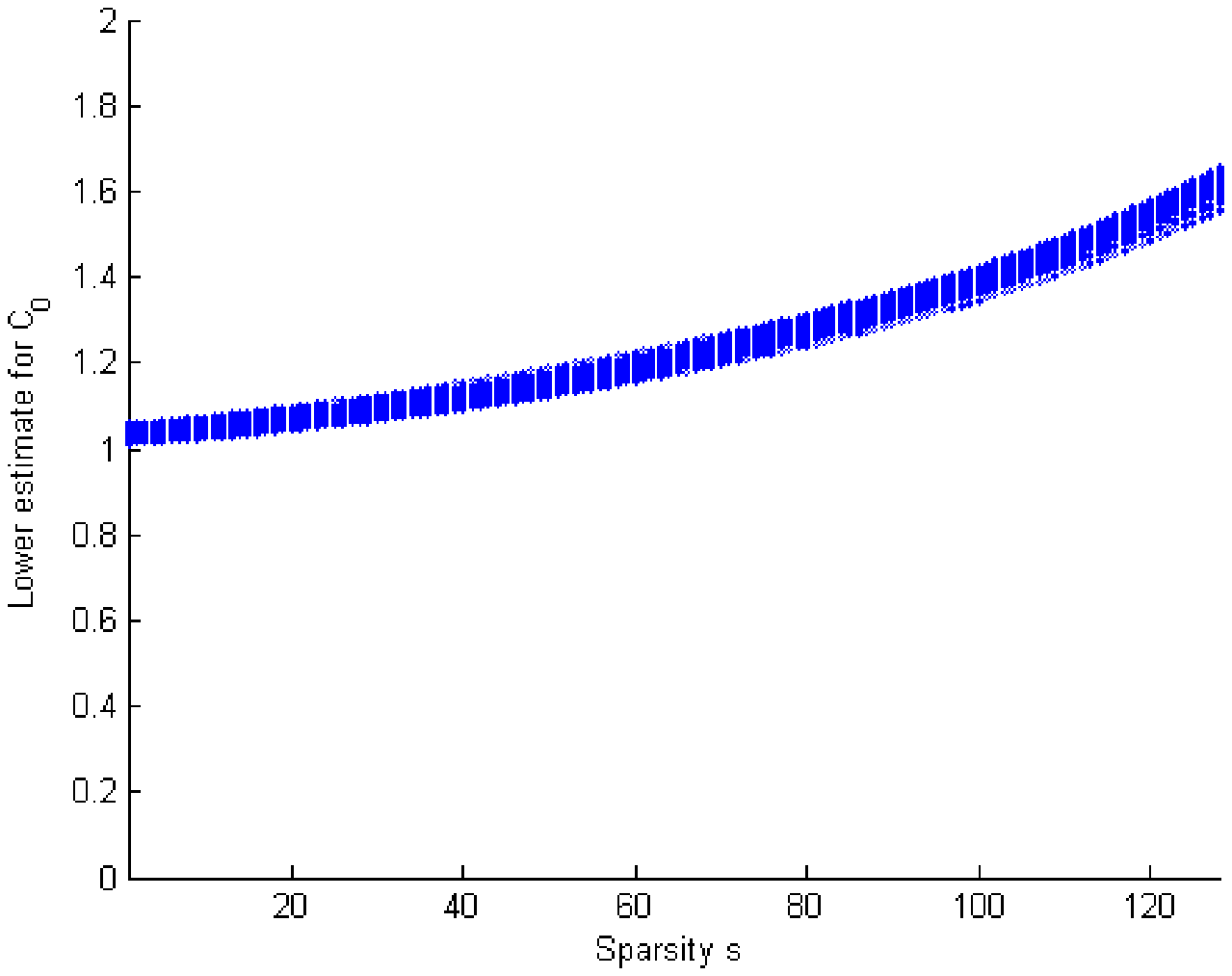}\\ (a)
		\fi
 	\end{minipage}
 	\begin{minipage}{7.5cm}
 	\centering
 		\ifpdf
		  \includegraphics[width=7.5cm]{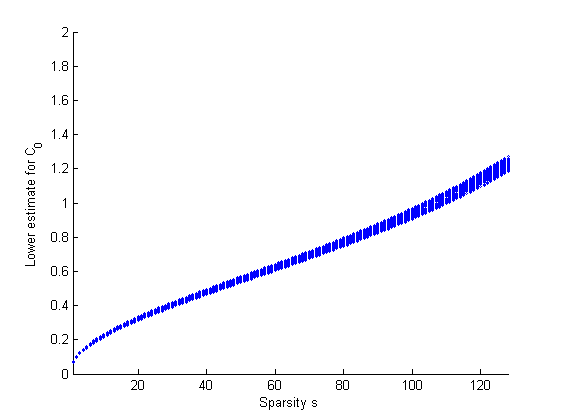}\\ (b)
		\else
		  \includegraphics[width=7.5cm]{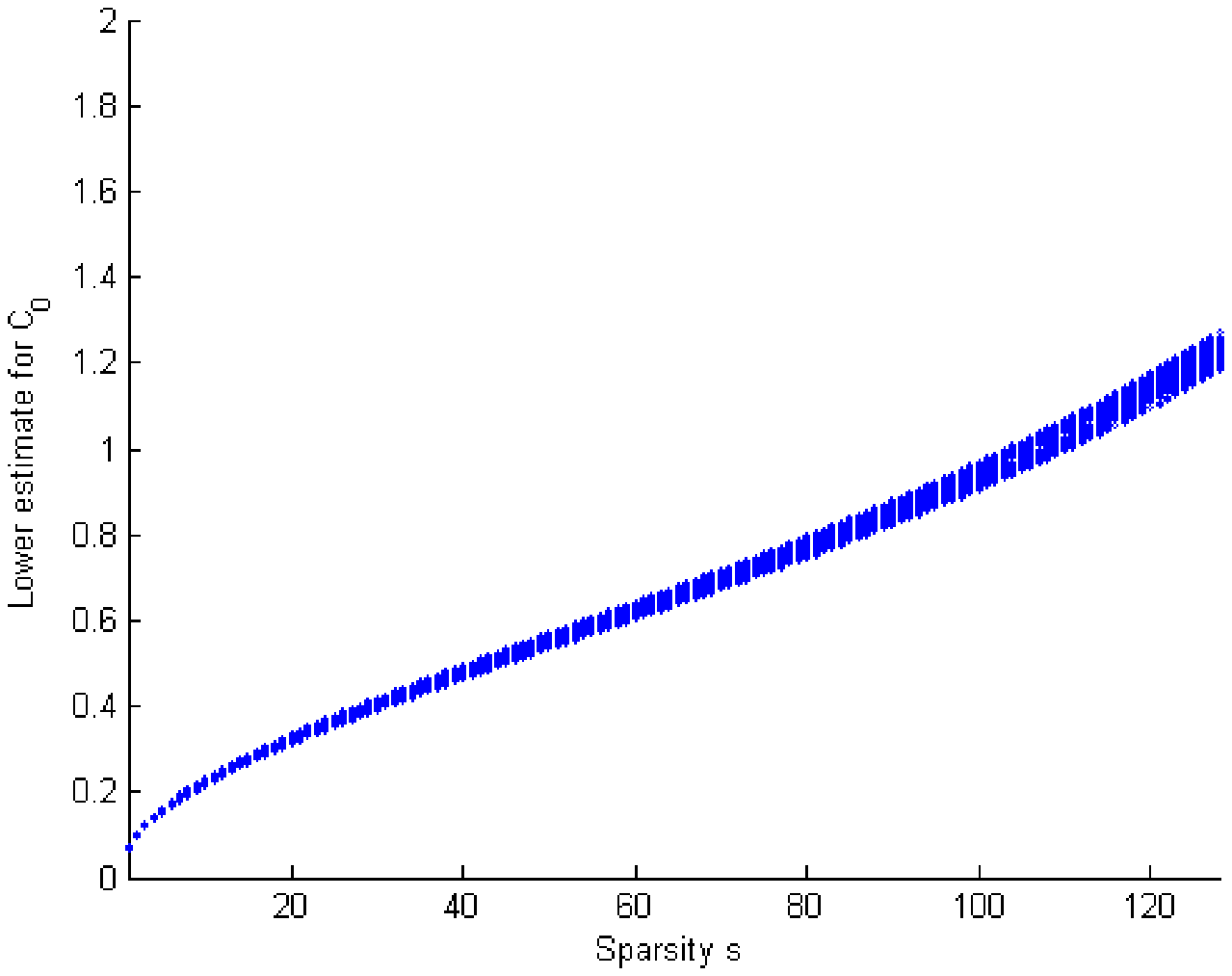}\\ (b)
		\fi
 	\end{minipage}
	\caption{\label{fig:02} Lower estimates for the constant $C_0$ in Corollary~\ref{l1mincor}, experiment performed for $n=256$ and $m=32$, results for $s=1,\ldots,128$. (a)~Results obtained from inequality~\eqref{exact-ineq1}. (b)~Results obtained from inequality~\eqref{exact-ineq2}.}
\end{figure}

\bigskip\bigskip
\noindent
\textbf{Acknowledgments}

The~research was supported by the~National Science Centre (Poland) based on the~decision DEC-2011/03/B/ST7/03649.

\end{document}